\documentclass{article}
\usepackage{graphicx}
\usepackage{amsmath}
\usepackage{amsfonts}
\usepackage{amssymb}

\newtheorem{theorem}{Theorem}

\newtheorem{definition}[theorem]{Definition}

\newtheorem{lemma}[theorem]{Lemma}

\newtheorem{remark}[theorem]{Remark}
\newenvironment{proof}[1][Proof]{\textbf{#1.} }{\hfill$\Box$}
\newenvironment{proof*}[1][Proof of the Theorem]{\textbf{#1.} }{\hfill$\Box$}
\begin{document}
\title{Positive Solutions for the $p$-Laplacian with Dependence on the Gradient}
\author{\textbf{H. Bueno, G. Ercole, W.M. Ferreira and A. Zumpano}\thanks{The authors were supported in part by FAPEMIG and CNPq-Brazil.}\\\textit{{\small Departamento de Matem\'{a}tica}},\\
\textit{{\small Universidade Federal de Minas Gerais, Belo
Horizonte, 30.123-970, Brazil}}}
\date{}
\maketitle

\begin{abstract}
We prove a result of existence of positive solutions of the Dirichlet problem for $-\Delta_p u=\mathrm{w}(x)f(u,\nabla u)$ in a bounded domain $\Omega\subset\mathbb{R}^N$, where $\Delta_p$ is the $p$-Laplacian and $\mathrm{w}$ is a weight function. As in previous results by the authors, and in contrast with the hypotheses usually made, no asymptotic behavior is assumed on $f$, but simple geometric assumptions on a neighborhood of the first eigenvalue of the $p$-Laplacian operator. We start by solving the problem in a radial domain by applying the Schauder Fixed Point Theorem and this result is used to construct an ordered pair of sub- and super-solution, also valid for nonlinearities which are super-linear both at the origin and at $+\infty$. We apply our method to the Dirichlet problem $-\Delta_pu = \lambda u(x)^{q-1}(1+|\nabla u(x)|^p)$ in $\Omega$ and give examples of super-linear nonlinearities which are also handled by our method.
\end{abstract}

\section{Introduction}
It is usually said that the sub- and super-solution method does not handle problems which are superlinear at the origin. One of the main purposes of this paper is to prove that this is not true.

Furthermore, considering the eigenvalues of the natural operator defined by the equation, we believe that imposing asymptotic conditions on the nonlinearity masks one of the main problems in differential equations, which is to completely understand how the crossing of the eigenvalues by the nonlinearity determines the solutions of the equation. This paper presents a contribution in this direction.

For this, we consider the Dirichlet problem
\begin{equation}\label{prob1}
{\displaystyle \left\{\begin{array}{rcll}
-\Delta_pu &=& \omega(x)f(u,|\nabla u|) &\mbox{in} \ \Omega, \\
u& =& 0 &\mbox{on} \ \partial\Omega\\
\end{array}\right.}
\end{equation}
where $\Omega \subset \mathbb{R}^N$ ($N>1$) is a smooth, bounded domain, $\Delta_pu:=\operatorname{div}\left(|\nabla u|^{p-2}\nabla u\right)$ is the $p$-Laplacian, $1<p<\infty$, $\omega\colon \overline{\Omega} \rightarrow \mathbb{R}$ is a continuous, nonnegative function with isolated zeros (which we will call \textit{weight function}) and
the $C^1$-nonlinearity $f\colon[0,\infty) \times [0,\infty) \rightarrow [0,\infty)$ satisfies simple hypotheses.

We solve (\ref{prob1}) for a large class of functions $f$, including nonlinearities that are super-linear both at the origin and at $+\infty$. (The continuous function $\omega$ has isolated zeroes only to simplify the presentation. It is enough that $\omega(x_0)>0$ for some $x_0\in\Omega$.)

We apply our approach to prove the existence of positive solution for the problem
\begin{equation*}
{\displaystyle \left\{\begin{array}{rcll}
-\Delta_pu &=& \lambda u(x)^{q-1}(1+|\nabla u(x)|^p) &\mbox{in} \ \Omega, \\
u& =& 0 &\mbox{on} \ \partial\Omega,\\
\end{array}\right.}\end{equation*}
and, in the sequence, we give examples of super-linear nonlinearities (both at the origin and at $+\infty$) which are also handled by our method.

Adapting methods and techniques developed in \cite{GZ}, where the nonlinearity does not depend on
$\nabla u$, we start by obtaining a radial, positive solution $u$ for the problem
\begin{equation}\label{prob2}{\displaystyle \left\{\begin{array}{rcll}
-\Delta_pu &=& \omega_\rho(|x-x_0|)f(u,|\nabla u|) &\mbox{in} \ B_\rho, \\
u& =& 0 &\mbox{on} \ \partial B_\rho,\\
\end{array}\right.}
\end{equation}
where $B_\rho$ is the ball with radius $\rho$ centered at $x_0$ and $\omega_\rho$ a \textit{radial} weight function. For this, no asymptotic behavior on $f$ is assumed but, instead, simple local hypotheses on the nonlinearity $f$. (See hypotheses (H1) and (H2) in the sequence.) The application of the Schauder Fixed Point Theorem yields a radial solution $u$ of (\ref{prob2}).

To cope with the general case of a smooth, bounded domain $\Omega$, we apply the method of sub- and super-solution as developed in \cite{BMP} (see Theorem \ref{TBoc} in the next section), for the general problem
\begin{equation}{\displaystyle \left\{\begin{array}{rcll}
-\Delta_pu &=& f(x,u,|\nabla u|) &\mbox{in} \ \Omega, \\
u& =& 0 &\mbox{on} \ \partial \Omega,\\
\end{array}\right.}\label{general}
\end{equation}
where $f\colon\Omega\times\mathbb{R}\times\mathbb{R}^N\to\mathbb{R}$, $(x,u,v)\mapsto f(x,u,v)$ is a Carath\'{e}odory function (i.e., measurable in the $x$-variable and continuous in the $(u,v)$-variable) satisfying
\begin{enumerate}
\item [(H3)] $f(x,u,v)\leq C(|u|)(1+|v|^p)\quad (u,v)\in\mathbb{R}\times\mathbb{R}^N,\ { a.e. }\ x\in\Omega$ for some increasing function $C\colon[0,\infty]\to[0,\infty]$.
\end{enumerate}
We observe that (H3) is also found in papers that do not apply the sub- and super-solution method (see \cite{figubilla, RUIZ}), since they are also related to the regularity of a weak solution.

Besides the Bernstein-Nagumo type assumption (H3), our hypotheses on the nonlinearity $f$ are not usual in the literature: we assume that $f$ has a \textit{local} behavior satisfying hypotheses of the type
\begin{enumerate}
\item [(H1)] $0 \leq f(u,|v|) \leq k_1 M^{p-1}, \
\textrm{if } \ 0 \leq u \leq M, \ |v| \leq \gamma  M$,
\item [(H2)] $f(u,|v|) \geq k_2{\delta}^{p-1}, \
\textrm{if } \  0<\delta \leq u \leq M, \ |v| \leq \gamma  M$,
\end{enumerate}
where the constants $k_1$, $k_2$ and $\gamma$ are defined later on in this paper and $\delta,M$ are arbitrary. These constants depend strongly on the weight function $\omega$ and in some special cases (for example, $\omega\equiv 1$) can be explicitly calculated, see Subsection \ref{ck1k2}. In \cite{HGZ} was proved that $k_1<\lambda_1<k_2$, where $\lambda_1$ stands for the first eigenvalue of the $p$-Laplacian.

Hypotheses (H1) and (H2) are geometrically interpreted in Figure \ref{fig2}.

\begin{figure}[tbh]
\begin{center}
{\includegraphics[height=4.5cm]{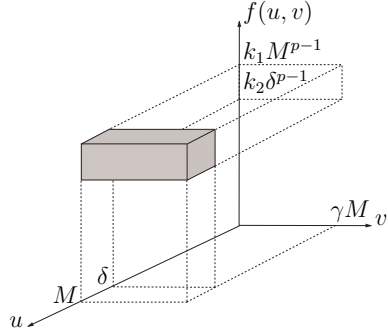}}
\end{center}
\caption{The graph of $f$ stays below $k_1 M^{p-1}$ in $[0,M] \times[0, \gamma M]$ and passes through the gray box.}\label{fig2}
\end{figure}

Hypotheses of this type will be considered in the scenarios of both the radial problem (\ref{prob2}) and the general problem (\ref{prob1}).

The radial problem (\ref{prob2}) is solved as an application of the Schauder Fixed Point Theorem. Therefore, a hypothesis like (H3) is not needed while studying this problem.

By considering a ball $B_\rho\subset\Omega$, radial symmetrization of the weight function $\omega$ permits us to consider a problem in the radial form in the sub-domain $B_\rho$, which has a solution $u_\rho$ as a consequence of our study of problem (\ref{prob2}). The chosen ball $B_\rho$ determines the value of the constants $k_2$ and $\gamma$ needed to solve (\ref{prob1}) and the radial solution $u_\rho\colon\overline{B_\rho}\to\mathbb{R}$ produces a sub-solution $\underline{u}$ of problem (\ref{prob1}), when we consider the extension $\underline{u}$ of $u_\rho$ defined by $\underline{u}(x)=0$, if $x\in \Omega\setminus B_\rho$. So, the solution of (\ref{prob2}) gives rise to a sub-solution of problem (\ref{prob1}).

In order to obtain a super-solution $\overline{u}$ for problem (\ref{prob1}), we impose that
\begin{equation}\label{quointrod}
\frac{\|\nabla \overline{u}\|_\infty}{\|\overline{u}\|_\infty}\leq\gamma,\end{equation}
an estimate that is suggested  by hypothesis (H1). So, we look for a super-solution of (\ref{prob1}) satisfying (\ref{quointrod}) and defined in a (smooth, bounded) domain $\Omega_2\supset \Omega$, which determines the value of the constant $k_1$ needed to solve (\ref{prob1}).

In the abstract setting of the domain $\Omega_2$, the super-solution $\overline{u}$ turns out to be a multiple of the solution $\phi_{\Omega_2}$ of the problem
\begin{equation}\label{pw}{\displaystyle \left\{\begin{array}{rcll}
-\Delta_p\phi_{\Omega_2} &=& \|\omega\|_\infty &\mbox{in} \ \Omega_2, \\
\phi_{\Omega_2}& =& 0 &\mbox{on} \ \partial \Omega_2,\\
\end{array}\right.}
\end{equation}
if $\phi_{\Omega_2}$ satisfies (\ref{quointrod}). In this setting, the existence of a positive solution for (\ref{prob1}) is stated in Section \ref{GD}.

We give two applications of this result for abstract nonlinearities in Section \ref{App}. In the first application, given in Subsection \ref{SSRadial}, we choose a ball $\Omega_2=B_R$ such that $\Omega\subset B_R$ and prove that, if $R$ is large enough, it is possible to obtain a super-solution for (\ref{prob1}) satisfying (\ref{quointrod}).

The second application is more demanding and considers the case where $\Omega_2$ is the domain $\Omega$ itself. In order to control the quotient (\ref{quointrod}), we assume $\Omega$ to be convex and apply a maximum result proved in Payne and Philippin \cite{PAYNE AND PHILIPPIN}. But, in some cases, if we choose $\Omega_2$ as the convex hull of $\Omega$, the same method produces a better solution than considering $\Omega\subset B_R$ for $R$ large enough.

In Section \ref{Ex} we consider concrete nonlinearities $f$. For $\lambda\in (0,\lambda^*]$ (where $\lambda^*$ is a positive constant), we apply the technique of Subsection \ref{SSRadial} and prove the existence of a positive solution for the problem
\begin{equation*}
{\displaystyle \left\{\begin{array}{rcll}
-\Delta_pu &=& f_\lambda(u, |\nabla u|) &\mbox{in} \ \Omega, \\
u& =& 0 &\mbox{on} \ \partial\Omega,\\
\end{array}\right.}\end{equation*}
where $f_\lambda(u, |\nabla u|):=\lambda u(x)^{q-1}(1+|\nabla u(x)|^p)$, $1<p<q$.

In the sequence, simple modifications in $f_\lambda(u, |\nabla u|)$ will produce new nonlinearities with are also handled by our method, including nonlinearities with superlinear behavior both at the origin and at $+\infty$.

\section{Comments on our method}
In general, variational techniques are not suitable to handle (\ref{prob1}); therefore, a combination of topological methods (as fixed-point or degree results) and blow-up arguments are usually applied to solve it (\cite{figubilla,iturriaga,RUIZ}). In the Laplacian case $p=2$, a combination of the mountain pass geometry with the contraction lemma was first used in \cite{DJAIRO}: an iteration process is constructed by freezing the gradient in each iteration and (variationally) solving the resulting problem. Then, Lipschitz hypotheses in the variables $u$ and $v$ are made on $f(x,u,v)$ in order to guarantee the convergence in $W_0^{1,2}(\Omega)$ of the obtained sequence of solutions. The same approach for the $p$-Laplacian with $p>2$ is not directly adaptable, since the natural extension of the Lipschitz conditions used to obtain the convergence of the iterated solutions yield a H\"{o}lder function $f$ with exponent greater than $1$ in variables $u$ and $v$.

When the nonlinearity $f$ does not depend on the gradient, the same technique was generalized in \cite{HGZ} to a smooth, bounded domain $\Omega\subset\mathbb{R}^N$. However, if $\Omega$ is not a ball, the dependence of $f$ on $|\nabla u|$ demands controlling  $\|\nabla u\|_\infty$ in $\Omega$ and complicates the application of Schauder's Fixed Point Theorem.

In \cite{figubilla}, the authors discuss the existence of positive solutions for
quasilinear elliptic equations in annular domains in $\mathbb{R}^N$ and, in particular,
the radial Dirichlet problem in annulus. (Therefore, the problem is transformed into an ordinary differential equation.)
In that paper, $f$ satisfies a super-linear condition at $0$ and a local super-linear condition at $+\infty$. The growth of the nonlinearity $f$ in relation to the gradient is controlled by a hypothesis similar to (H3) and a local homogeneity type condition in the second variable, hypothesis related to the behavior of $f$ near a point $(r,s,t)$ such that $f(r,s,t)=0$, where $r=|x|$. The existence of solutions is guaranteed by applying the Krasnosel'skii Fixed Point Theorem for mappings defined in cones.

The majority of papers dealing with the sub- and super-solution method with nonlinearities depending on the gradient are focused in the improvement of the method itself, that is, the papers aim to weaken the hypotheses of the method (\cite{KURA,LEON}). One exception is the paper of Grenon \cite{Grenon}, where problem (\ref{general}) is solved by analyzing two symmetrized problems. From the existence of two nontrivial super-solutions $V_1$ and $V_2$ for those problems follows the existence of a super-solution $U_1$ and a sub-solution $U_2$ for $(\ref{general})$, with $U_2 \leq U_1$.

More recently, the sub- and super-solution method has been applied to some instances of problem (\ref{prob1}): the Laplacian case $p=2$ and $\Omega=\mathbb{R}^N$. In \cite{GHERGU RADULESCU1,JOSE VALDO1}, the dependence on the gradient occurs by means of a convection term $|\nabla u|^\alpha$ in the nonlinearity $f$ and the authors look for ground state solutions. These are obtained as limits of a monotone sequence of auxiliary problems defined in nested subdomains of $\mathbb{R}^n$, which are bounded and smooth.



\section{Preliminaries}
In this section we recall some basic results in the theory of the $p$-Laplacian equation with Dirichlet boundary condition and present technical results that will be used in the rest of the paper. Let $D$ be a bounded, smooth domain in $\mathbb{R}^N$, $N>1$.
\begin{definition}\label{defsubsupsol}Let $f\colon D\times\mathbb{R}\times\mathbb{R}^N$ be a Carath\'{e}odory function. A function $u \in W^{1,p}(D)\cap\,L^{\infty}(D)$ is called a solution $($sub-solution, super-solution$)$ of
\begin{equation}\label{p1}
{\displaystyle \left\{\begin{array}{rcll}
-\Delta_pu &=& f(x,u,\nabla u) &\mbox{in} \ D, \\
u& =& 0 &\mbox{on} \ \partial D,\\
\end{array}\right.}
\end{equation}
if
\begin{equation*}
\displaystyle{\int_{D}|\nabla u|^{p-2}\nabla u\cdot\nabla\phi\,dx=
\int_{D}f(x,u,\nabla u)\phi\,dx \ \ (\leq 0,\ \geq 0),}
\end{equation*}
for all $\phi \in C_0^{\infty}(D)$, $(\phi\geq 0$ in $D$ in the case of a sub- or super-solution$)$ and
$$u=0 \ (\leq 0, \ \geq 0) \ \ \textrm{on} \ \ \partial D.$$

A pair $(\underline{u},\overline{u})$ of sub- and super-solution is \textrm{ordered} if $\underline{u} \leq
\overline{u}$ a.e.\goodbreak
\end{definition}

The hypothesis (H3) implies that
\begin{eqnarray*}
\int_{D}\big|f(x,u,\nabla u)\phi\big|dx &\leq& C(\|u\|_\infty)\int_{D}(1+|\nabla u|^p)|\phi|\, dx\\
&=& C(\|u\|_\infty)\left(\int_{D}|\phi|\, dx+\int_{D}|\nabla u|^p|\phi|\,dx\right)< \infty,
\end{eqnarray*}
since $\phi \in C_0^{\infty}$ and $u \in W^{1,p}( D) \cap\, L^{\infty}(D)$.\vspace{.3cm}

We now state, in a version adapted to our paper, the result that give basis to the method of sub- and super-solution for equations like (\ref{p1}). The existence part is a consequence of Theorem 2.1 of Boccardo, Murat and Puel \cite{BMP}. The regularity part follows from the estimates of Lieberman \cite{LIEBERMAN}, while the minimal and maximal solutions are consequence of Zorn's Lemma, as proved in Cuesta Leon \cite{mabel}:
\begin{theorem}\label{TBoc} Let $f\colon\Omega\times\mathbb{R}\times\mathbb{R}^N\to\mathbb{R}$ be a Carath\'{e}odory function satisfying $({\rm H}3)$. Suppose that $(\underline{u},\overline{u})$ is an ordered pair of sub- and super-solution for the problem $(\ref{p1})$.

Then, there exists a minimal solution $u$ and a maximal solution $v$ of $(\ref{p1})$, both in $C^{1,\tau}\left(\overline{\Omega}\right)$ $(0<\tau<1)$, such that $\underline{u}\leq u\leq v\leq\overline{u}$.
\end{theorem}
(By \emph{minimal} and \emph{maximal} solution of $(\ref{p1})$ we mean that, if $w$ is a solution of this problem and $\underline{u}\leq w\leq \overline{u}$, then $u\leq w\leq v$.)

\begin{remark}\label{regularity}\rm It is well-known that, for each $h \in L^\infty(D)$, the problem
\begin{equation}\label{hdiric}
{\displaystyle \left\{\begin{array}{rcll}
-\Delta_pu &=& h &\mbox{in} \ D, \\
u& =& 0 &\mbox{on} \ \partial D\\
\end{array}\right.}
\end{equation}
has a unique weak solution $u \in W_0^{1,p}(D)$, which belongs to $C^{1,\tau}(\overline{D})$ for some $0 <\tau <1$.
\end{remark}

If we assume that $f\colon \overline{D}\times\mathbb{R}\times\mathbb{R}^N\to\mathbb{R}$ is continuous, then is compact and continuous the operator $T\colon C^1\left(\overline{D}\right)\to C^{1}\left(\overline{D}\right)$ defined by $Tu=v$, where $v$ is the solution of
\begin{equation}\label{p}
{\displaystyle \left\{\begin{array}{rcll}
-\Delta_pv&=& f(x,u,\nabla u) &\mbox{in} \ D, \\
v& =& 0 &\mbox{on} \ \partial D.\\
\end{array}\right.}
\end{equation}
(The space $C^1\left(\overline{D}\right)$ is the Banach space of continuously differentiable functions endowed with the norm $\|u\|_1=\|u\|_\infty+\|\nabla u\|_\infty$.)


\begin{theorem}\label{tcomp}$T$ is continuous and compact.
\end{theorem}
\begin{proof}
It is clear that the mapping $T_1\colon C^1(\overline{D})\to L^\infty(\overline{D})$ given by $T_1(u)=f(x,u,\nabla u)$ is continuous. As pointed out in \cite{Azizieh}, Lemma 1.1, the $L^\infty$-estimates of Anane \cite{Anane} and the $C^{1,\alpha}$-estimates of Liebermann \cite{LIEBERMAN} and Tolksdorf \cite{Tolksdorf} imply that $T_2\colon L^\infty(\overline{D})\to C^1(\overline{D})$, defined by $T_2(h)=v$, where
\[{\displaystyle \left\{\begin{array}{rcll}
-\Delta_pv&=& h &\mbox{in} \ D, \\
v& =& 0 &\mbox{on} \ \partial D.\\
\end{array}\right.}\]
is compact and continuous.

Since $T=T_2\circ T_1$, where are done.
\end{proof}\vspace{.5cm}

Problem (\ref{hdiric}) satisfies a comparison principle and a strong maximum principle. (See \cite{DAMASCELLI},
Thm 1.2 and Thm 2.2, respectively.) It follows easily a comparison principle between solutions defined in different domains:
\begin{lemma}\label{cpdd} Suppose that $\Omega_1$, $\Omega_2$ are smooth
domains, $\Omega_1 \subset \Omega_2$. For $i \in \{1,2\}$, let
$L^\infty\ni h_i\colon \Omega_i \rightarrow \mathbb{R}$ and
$u_i \in C^{1,\alpha}(\overline{\Omega_i})$ be the weak solution of
the problem\begin{equation}{\displaystyle \left\{
\begin{array}{rcll}-\Delta_p u_i &=& h_i &\mbox{in} \ \Omega_i, \\u_i& =& 0 &\mbox{on} \ \partial \Omega_i.\\\end{array}\right.}\end{equation}Let us suppose that
\begin{enumerate}
\item [$(i)$]$0 \leq h_1 \leq h_2$ in $\Omega_1$,\item [$(ii)$]$u_1 \leq u_2$ in $\partial \Omega_1$.\end{enumerate}

Then $u_1\leq u_2$  in $\Omega_1$.
\end{lemma}

In the setting of equation (\ref{hdiric}), we define
\begin{equation}\label{k1}
k_1(D):=\|\phi_D\|_{\infty}^{-(p-1)}
\end{equation}
where $\phi_D \in C^{1,\alpha}(\overline{D})\cap W_0^{1,p}(D)$ is the solution of
\begin{equation}\label{problema com w(x) em omega} {\displaystyle \left\{\begin{array}{rcll}
-\Delta_p \phi_D &=& \omega_D &\mbox{in} \ D, \\
\phi_D& =& 0 &\mbox{on} \ \partial D.\\
\end{array}\right.}
\end{equation}
By the maximum principle, $\phi_D > 0$ in $D$ and $k_1(D)$ is well defined.
\begin{remark}\label{rcpdd}\rm In the context of Lemma \ref{cpdd}, that is, $\Omega_1\subset\Omega_2$ and $\omega_{\Omega_1}\leq\omega_{\Omega_2}$, it follows immediately that
\[k_1(\Omega_2)=\|\phi_{\Omega_2}\|_\infty^{-(p-1)} \leq\|\phi_{\Omega_1}\|_\infty^{-(p-1)}= k_1(\Omega_1).\]
\end{remark}\vspace{.4cm}

In the special case $D=B_\rho$, a ball of radius $\rho$ centered in $x_0 \in \Omega$, let us consider the Dirichlet problem
\begin{equation}\label{prad}
{\displaystyle \left\{\begin{array}{rcll}
-\Delta_p \phi_\rho &=& \omega_\rho(|x-x_0|) &\text{in} \ B_\rho, \\
\phi_\rho& =& 0 &\text{on} \ \partial B_\rho,\\
\end{array}\right.}\end{equation}
where $\omega_\rho\colon\overline{B_\rho} \rightarrow \mathbb{R}$ is a \emph{radial} weight function.

It is straightforward to verify that the solution of (\ref{prad}) is given by
\begin{equation}\label{sprad}
\phi_\rho(|x-x_0|)=\displaystyle{\int_{|x-x_0|}^\rho \left(\int_0^\theta K(s,\theta) ds\right)^{\frac{1}{p-1}} d\theta}, \ |x-x_0| \leq \rho,\end{equation}
where
\begin{equation}\label{kernel}
K(s,\theta)=\left(\frac{s}{\theta}\right)^{N-1} \omega_\rho(s).\end{equation}

The solution $\phi_\rho$ satisfies $\phi_\rho \in C^2(\overline{B_\rho})$ if $1<p \leq 2$ and $\phi_\rho \in
C^{1,\tau}(\overline{B_\rho})$ if $p>2$, where $\tau=1/(p-1)$. (See \cite{HGZW}, Lemma 2 for details.)

We also define another constant that will play an essential role in our technique:
\begin{eqnarray}\label{k2}
k_2(B_\rho)&=&\displaystyle{\left[\int_t^\rho\left(\int_0^t K(s,\theta)ds\right)^{\frac{1}{p-1}}d\theta \right]^{1-p}}\notag\\ &=&\displaystyle{\left[\max_{0 \leq r \leq \rho}\int_r^\rho \left(\int_0^r K(s,\theta)ds\right)^{\frac{1}{p-1}}d\theta\right]^{1-p}}.\end{eqnarray}

Since $\omega_\rho$ has isolated zeroes and the function
\begin{equation*}\label{alpha - t maior que zero} \alpha \rightarrow \displaystyle{ \int_\alpha^\rho
\left(\int_0^\alpha K(s,\theta)ds\right)^{\frac{1}{p-1}}d\theta}\end{equation*}
is nonnegative and vanishes both at $\alpha=0$ and at $\alpha=\rho$, we have $t>0$.\vspace{.3cm}

We now establish the relation between $k_1(D)$ and $k_2(B_\rho)$, also valid in the case $D=B_\rho$:
\begin{lemma}\label{lema comparacao entre k1 de omega e k2 da bola}
Let $D$ be a smooth domain in $\mathbb{R}^N$ $(N>1)$, $B_\rho\subseteq D$ a ball of center $x_0$ and radius $\rho>0$ and $k_1(D),\ k_2(B_\rho)$ the constants defined by $(\ref{k1})$ and $(\ref{k2})$, respectively, where $\omega_\rho$ is a radial weight function such that $\omega_D\geq \omega_\rho$ in $B_\rho$.

Then, $k_1(D)<k_2(B_\rho)$.
\end{lemma}
\begin{proof}We have
\[{\displaystyle \left\{
\begin{array}{rcll}
-\Delta_p \phi_D=\omega_D& \geq & \omega_\rho= -\Delta_p \phi_\rho &\mbox{in} \ B_\rho, \\
\phi_D& \geq & 0=\phi_\rho &\mbox{on} \ \partial B_\rho\\
\end{array}\right.}\]
and the comparison principle yields $\phi_D\geq \phi_\rho$ in $B_\rho$.

Therefore
$$\displaystyle{\|\phi_D\|_\infty \geq \| \phi_\rho\|_\infty= \int_0^\rho
\left(\int_0^{\theta} K(s,\theta) ds \right)^{\frac{1}{p-1}} d\theta}$$
and
$$k_1(D)^{-1}=\|\phi_D\|_{\infty}^{p-1} \geq \|\phi_\rho\|_{\infty}^{p-1}= \displaystyle{\left[ \int_0^\rho
\left(\int_0^\theta K(s,\theta)ds\right)^{\frac{1}{p-1}}d\theta
\right]^{p-1}}.$$

Since the zeroes of $\omega_D$ are isolated and $t \neq 0$, we have
$$ k_1(D)^{-1}> \displaystyle{\left[ \int_t^\rho \left(\int_0^t K(s,\theta)ds\right)^{\frac{1}{p-1}}d\theta
\right]^{p-1}=k_2(B_\rho)^{-1}},$$
proving our result.
\end{proof}
\section{Radial Solutions}\label{SRad}
In this section we study the radial version of (\ref{prob1}), that is
\begin{equation*}\hspace{2.1cm}{\displaystyle\left\{\begin{array}{rcll}
-\Delta_pu &=& \omega_\rho(|x-x_0|)f(u,|\nabla u|) &\mbox{in} \ B_\rho, \\
u& =& 0 &\mbox{on} \ \partial B_\rho,\\
\end{array}\right.}\hspace{2.1cm}(\ref{prob2})
\end{equation*}
where $B_\rho$ is a ball of radius $\rho$ centered in $x_0$ and $\omega_\rho\colon\overline{B_\rho} \rightarrow \mathbb{R}$ is a radial weight function.

A solution of (\ref{prob2}) will be obtained by applying the Schauder Fixed Point Theorem in the space $C^1(B_\rho)$. So, the hypothesis (H3) is not necessary; we only need $f$ to be continuous.

The radial boundary value problem equivalent to (\ref{prob2}) is
\begin{equation}\label{rp}
\left\{\begin{array}{rcll}
{\displaystyle\frac{d}{dr}}\left(-r^{N-1}\varphi_p(u'(r))\right)&=&r^{N-1}\omega_\rho(r)f(u,|u'(r)|),\quad 0<r<\rho\\
u'(0)&=&0,\\
u(\rho)&=&0,&
\end{array}\right.\end{equation}
where $\varphi_p(\xi)=|\xi|^{p-2}\xi$ for $1<p<\infty$.

If $q=p/p-1$ and $u>0$, the function $\varphi_q$, inverse of $\varphi_p$, is given by
\[\varphi_q(u)=|u|^{q-2}u=u^{q-1}=u^{\frac{p}{p-1}-1}=u^{\frac{1}{p-1}}.\]

We remark that the function $\phi_p(r)$, $r=|x-x_0|$, defined by (\ref{sprad}), can be written as
\[\phi_\rho(r)=\int_r^\rho \varphi_p\left(\int_0^\theta K(s,\theta)\,ds\right)d\theta.\]
So,
\[\phi'_\rho(r)=-\varphi_q\left(\int_0^r K(s,r)\,ds\right)\]
and $|\nabla \phi_\rho(|x-x_0|)|=|\phi'_\rho(r)\frac{x-x_0}{r}|=|\phi'_\rho(r)|$. Therefore, we have $\|\nabla\phi_\rho\|_\infty=\max_{0\leq r\leq \rho}|\phi'_\rho(r)|$.

To prove the existence of solutions for problem (\ref{prob2}), we suppose the existence of $\delta$ and $M$, with $0<\delta<M$ such that the nonlinearity $f$ satisfies
{\renewcommand{\leftmargini}{1cm}
\begin{enumerate}
\item [(H$1_r)$] \label{H2a} $0 \leq f(u,|v|) \leq k_1(B_\rho) M^{p-1}, \
\textrm{if } \ 0 \leq u \leq M, \ |v| \leq \gamma_\rho  M$;
\item [(H$2_r)$] \label{H2b} $f(u,|v|) \geq k_2(B_\rho){\delta}^{p-1}, \
\textrm{if } \  \delta \leq u \leq M, \ |v| \leq \gamma_\rho  M$,
\end{enumerate}}
\noindent with $k_1(B_\rho)$ and $k_2(B_\rho)$ defined by (\ref{k1}) and (\ref{k2}), respectively, and $\gamma_\rho$ defined by
\begin{equation}\label{gamma}
\gamma_\rho=\max_{0\leq r \leq \rho}\varphi_q \left(k_1(B_\rho)\int_0^{r} K(s,r)\, ds \right)=\frac{\|\nabla \phi_\rho\|_\infty}{\|\phi_\rho\|_\infty}.
\end{equation}
We remark that $k_1(B_\rho)$, $k_2(B_\rho)$ and $\gamma_\rho$ depend only on $\rho$ and $\omega_\rho$. The hypothesis (H$2_r$) aims to discard $u\equiv 0$ as a solution of (\ref{prob2}), in the case $f(0,|v|)=0$.

We also define the continuous functions $\Psi_{\delta}$, $ \Phi_M $ and
$\Gamma_M$ by
\begin{equation}\label{Psi delta}
{\displaystyle \Psi_{\delta}(r)=\left\{\begin{array}{ll}\delta, &\mbox{if} \ 0 \leq r \leq t, \\
\displaystyle{\delta \int_r^\rho \varphi_q \left(k_2(B_\rho) \int_0^{t} K(s,\theta)\, ds \right) d\theta},& \mbox{if} \ t <r \leq \rho,\\ \end{array}\right.}
\end{equation}
\begin{eqnarray}\label{Phi M}
\Phi_M(r)&=&\displaystyle{M \int_r^\rho \varphi_q \left(k_1(B_\rho)\int_0^{\theta} K(s,\theta)\, ds
\right) d\theta}\notag\\
&=& M\frac{\phi_\rho(r)}{\|\phi_\rho\|_\infty},\ \textrm{if} \ 0<r\leq \rho,\end{eqnarray}
and
\begin{eqnarray}\label{Gamma M}
\Gamma_M(r)&=&M \varphi_q \left( k_1(B_\rho)\int_0^r K(s,r)\, ds\right)\notag\\
 &=&M\frac{|\phi_\rho'(r)|}{\|\phi_\rho\|_\infty},\  \textrm{if} \ 0<r\leq \rho.\end{eqnarray}
\begin{lemma}\label{lemmaaux}We have
\begin{enumerate}
\item[$(i)$] $0 \leq \Phi_M(r) \leq M$;
\item[$(ii)$] $0 \leq \Gamma_M(r)\leq \gamma_\rho M$;
\item[$(iii)$] $0 \leq \Psi_{\delta}(r)\leq \Phi_M(r)$.
\end{enumerate}
\end{lemma}
\begin{proof}
The items $(i)$ and $(ii)$ are obvious.

It follows from (H$1_r$) and (H$2_r$) that $k_2(B_\rho) \delta^{p-1}\leq k_1(B_\rho)
M^{p-1}$. Therefore, if $0 \leq r \leq t$, then
\begin{align*}
\Phi_M(r)&=M\int_r^\rho\varphi_q\left(k_1(B_\rho)\int_0^\theta K(s,\theta)\,ds\right)
d\theta\\
&\geq \int_t^\rho \varphi_q\left(k_1(B_\rho)M^{p-1}\int_0^tK(s,\theta)\, ds\right)d\theta\\
&\geq\int_t^\rho\varphi_q\left(k_2(B_\rho)\delta^{p-1}\int_0^t K(s,\theta)\,ds\right)d\theta\\
&=\delta=\Psi_\delta(r).
\end{align*}

If $t < r \leq \rho$, then
\begin{align*}
\Phi_M(r)&=M\int_r^\rho\varphi_q\left(k_1(B_\rho)\int_0^\theta K(s,\theta)\,ds\right)d\theta\\
&\geq\int_r^\rho\varphi_q\left(k_2(B_\rho)\delta^{p-1}\int_0^\theta K(s,\theta)\,ds\right)d\theta\\
&\geq\delta\int_r^\rho\varphi_q\left(k_2(B_\rho)\int_0^t K(s,\theta)\,ds\right)d\theta= \Psi_\delta(r),
\end{align*}
completing the proof of $(iii)$.
\end{proof}

We now establish the main result of this section:
\begin{theorem}\label{sball} Suppose that the continuous nonlinearity $f$ satisfies $({\rm H}1_r)$ and $({\rm H}2_r)$. Then the problem
\begin{equation}{\displaystyle \left\{\begin{array}{rcll}
-\Delta_pu &=& \omega_\rho(|x-x_0|)f(u,|\nabla u|) &\mbox{in} \ B_\rho, \\
u& =& 0 &\mbox{on} \ \partial B_\rho,\\
\end{array}\right.}\tag{\ref{prob2}}
\end{equation}
has at least one positive solution $u_\rho(|x-x_0|)$ satisfying
$$\Psi_\delta \leq u_\rho \leq \Phi_M \ and \ |\nabla u_\rho|\leq \Gamma_M$$
$($and so $\delta \leq \|u_\rho\|_\infty \leq M$ and $\|\nabla u_\rho\|_\infty\leq\gamma_\rho M)$.
\end{theorem}
\begin{proof}To obtain a positive solution of (\ref{rp}), we consider the
Banach space $X=C^1([0,\rho])$, with the norm $\displaystyle{\|u\|=\sup_{s\in[0,\rho]}|u(s)| +\sup_{s\in[0,\rho]}|u'(s)|}$ and
the integral operator $T\colon X\to X$ defined by
\begin{equation*}\label{U} \displaystyle{(Tu)(r)=\int_r^\rho \varphi_q \left(\int_0^\theta K(s,\theta)\, f(u(s),| u'(s)|)ds\right)\,d\theta}, \ 0\leq r \leq \rho.\end{equation*}

It follows immediately that
\begin{equation*}\displaystyle{(Tu)'(r)=-\varphi_q \left(\int_0^r K(s,r) f(u(s),|u'(s)|)\,ds\right)}, \ 0\leq
r \leq \rho.\end{equation*}

Theorem \ref{tcomp} yields the continuity and compactness of $T$. (A direct proof that $T\colon X \to X$ is a continuous, compact operator can be found in the Appendix.)

Now we observe that, if $u$ is a fixed point of the operator $T$, then $u$ is a solution of (\ref{rp}). To prove the existence of a fixed point $u$ of $T$, we apply the Schauder Fixed Point Theorem in the closed, convex and bounded subset
\begin{equation*}\label{Y}
Y=\left\{ u \in X: \Psi_{\delta}\leq u \leq \Phi_M \ \textrm{and}
\ |u'| \leq \Gamma_M \right\}.
\end{equation*}

We need only to show that $T(Y) \subset Y$.

It follows from Lemma \ref{lemmaaux} and (H$1_r$) that, for all
$0\leq r \leq \rho$, we have
\begin{align*}
(Tu)(r)&=\int_r^\rho\varphi_q\left(\int_0^\theta K(s,\theta)f(u(s), |u'(s)|) ds\right)\,d \theta\\
&\leq\int_r^\rho\varphi_q\left(k_1(B_\rho)M^{p-1}\int_0^\theta K(s,\theta)\,ds\right)d\theta\\
&=M\int_r^\rho \varphi_q \left( k_1(B_\rho)\int_0^\theta K(s,\theta)\, ds \right)d \theta\\
&= \Phi_M(r)\end{align*} and
\begin{align*}
|(Tu)'(r)| & = \varphi_q \left(\int_0^r K(s,r) f(u(s), |u'(s)|)\, ds \right)\\
& \leq \varphi_q \left(k_1(B_\rho) M^{p-1} \int_0^r K(s,r)\,ds \right)\\
& = M \varphi_q \left(  k_1(B_\rho) \int_0^r K(s,r)\,ds \right)\\
& = \Gamma_M(r) \leq \gamma_\rho  M.\end{align*}

Suppose that $0 \leq r \leq t$. The definition of $Y$ and (H$2_r$) imply that \[f(u(s),|u'(s)|)\geq k_2(B_\rho)\delta^{p-1}.\] Therefore,
\begin{align*}
(Tu)(r) & =\int_r^\rho\varphi_q \left( \int_0^\theta K(s,\theta) f(u(s), |u'(s)|)\, ds \right)d \theta\\
& \geq \int_t^\rho \varphi_q \left(\int_0^t K(s,\theta) f(u(s), |u'(s)|)\, ds \right)d \theta\\
& \geq \int_t^\rho \varphi_q \left( k_2(B_\rho) \delta^{p-1} \int_0^t K(s,\theta)\,ds \right)d \theta\\
&= \delta= \Psi_\delta(r).
\end{align*}
If $t \leq r \leq \rho$, then
\begin{align*}
(Tu)(r) & =\int_r^\rho \varphi_q \left(\int_0^\theta K(s,\theta)f(u(s), |u'(s)|)\, ds \right)d \theta\\
& \geq \int_r^\rho \varphi_q \left(\int_0^t K(s,\theta)f(u(s), |u'(s)|)\, ds \right)d \theta\\
& \geq \int_r^\rho \varphi_q \left(  k_2(B_\rho) \delta^{p-1}\int_0^t K(s,\theta)\, ds \right)d \theta\\
& = \Psi_\delta(r).
\end{align*}

So, we have $T(Y) \subset Y$. By the Schauder Fixed Point Theorem, we conclude the existence of at least one positive solution $u_\rho$ for (\ref{rp}) in $Y$, thus implying that $u_\rho(|x-x_0|)$ is a positive solution of (\ref{prob2}) that satisfies the bounds stated in the theorem.
\end{proof}
\section{Existence of Solutions in General Domains}\label{GD}
In this section we state and prove our main abstract result: the existence of a positive solution
for
\begin{equation}\left\{\begin{array}{rcll}
-\Delta_pu &=& \omega(x)f(u,|\nabla u|) &\mbox{in} \ \Omega, \\
u& =& 0 &\mbox{on} \ \partial\Omega.
\end{array}\right.
\tag{\ref{prob1}}%
\end{equation}

We start by defining the parameters we need to formulate our hypotheses.

Let $\Omega_{2}$ be a bounded, smooth domain such that $\Omega_{2}%
\supset\Omega$ and define
\[k_1(\Omega_2):=\|\phi_{\Omega_2}\|_{\infty}^{1-p},\]
where $\phi_{\Omega_{2}}$ is the solution of
\begin{equation}{\displaystyle \left\{\begin{array}{rcll}
-\Delta_p\phi_{\Omega_2}&=&\|\omega\|_\infty &\mbox{in} \ \Omega_2, \\
\phi_{\Omega_2}& =& 0 &\mbox{on} \ \partial \Omega_2.\end{array}\right.}\label{pss}
\end{equation}

Now, for any ball $B_\rho\subset\Omega$ with center in $x_0\in\Omega$ and radius $\rho>0$, let us to
denote by $\omega_{\rho}$ the radial function defined by%
\begin{equation}
\omega_\rho(s)=\left\{\begin{array}{ll}%
\min\limits_{|x-x_0|=s}\omega(x), & \text{if }\ 0<s\leq\rho,\\
\omega(x_0), & \text{if }\ s=0.
\end{array}\right.\label{omegarho}
\end{equation}

Thus, by using this function we consider $k_2(B_\rho)$ and $\gamma_\rho$, defined in accordance to the former definitions (\ref{k2}) and (\ref{gamma}), respectively.

At last, we fix $\rho>0$ such that (see Remark \ref{rgamma}, below)
\begin{equation}\label{quo}
\frac{\|\nabla\phi_{\Omega_2}\|_{\infty}}{\|\phi_{\Omega_2}\|_\infty}\leq\gamma_\rho
\end{equation}
and then we set the parameters
\[k_1:=k_1(\Omega_2),\quad k_2:=k_2(B_\rho)\quad\text{and}\quad\gamma=\gamma_\rho.\]
\begin{theorem}\label{maint}
Suppose that, for arbitrary $\delta,M$ such that $0<\delta<M$, the nonlinearity $f$ satisfies:
\begin{enumerate}
\item[$({\rm H}1)$] $0\leq f(u,|v|)\leq k_1M^{p-1}$,\ if\ $0\leq u\leq M$, $|v|\leq\gamma M$;
\item[$({\rm H}2)$] $f(u,|v|)\geq k_2\delta^{p-1}$, \ if\ $\delta\leq u\leq M,$ $|v|\leq\gamma M$;
\item[$({\rm H}3)$] $f(u,|v|)\leq C(|u|)\left(1+|v|^p\right)$ for all $(x,u,v)$, where $C\colon[0,\infty)\rightarrow [0,\infty)$ is increasing.
\end{enumerate}

Then,
problem $(\ref{prob1})$ has a positive solution $u$ such that%
\[\delta\leq\|u\|_\infty\leq M\text{ in }\Omega.\]
\end{theorem}

\begin{remark}\label{rgamma}\rm We would like to observe that the inequality (\ref{quo}) always occurs, if $\rho$ is taken sufficiently small such that
\begin{equation}
\frac{\|\nabla\phi_{\Omega_2}\|_\infty}{\|\phi_{\Omega_2}\|_\infty}\leq\frac{1}{\rho}.\label{*}%
\end{equation}
In fact, we have the gross estimate
\[\frac{1}{\rho}\leq\gamma_{\rho},\ \text{ for any }B_{\rho}\subset\Omega
\]
since $\gamma_\rho=\frac{\|\nabla \phi_\rho\|_\infty}{\|\phi_\rho\|_\infty}$ and
\[\|\phi_\rho\|_\infty=\phi_\rho(0)=-\int_0^\rho\phi_\rho'(s)\,ds=\int_0^\rho |\phi_\rho'(s)|\,ds\leq \rho\|\nabla\phi_\rho\|_\infty.\]

We supposed that the weight function $\omega$ has isolated zeroes. As mentioned, this assumption is not necessary: it will only be used in the discussion about the best possible choice for the constants $k_1$ and $k_2$, which is done in Subsection \ref{SSRadial}.
\end{remark}\vspace{.5cm}

In Section \ref{App} we give examples of $\Omega_2$ and $\rho$ satisfying (\ref{quo}). There, we consider the cases $\Omega_2=B_R\supset\Omega$ and, supposing $\Omega$ convex, $\Omega_2=\Omega$ and present better estimates than (\ref{*}) to choose $\rho$.

The obtention of a sub-solution for problem (\ref{prob1}) is based on the following general result:
\begin{lemma}\label{lemasub}Let $\Omega$ and $\Omega_1$ be smooth domains in $\mathbb{R}^N$ $(N>1)$, with $\Omega_1 \subset \Omega$.
Let $u_1 \in C^{1,\tau}\left(\overline{\Omega_1}\right)$ be a positive solution of
\begin{equation*} 
{\displaystyle \left\{\begin{array}{rcll}
-\Delta_pu_1 &=& f_1(x,u_1,\nabla u_1) &\mbox{in} \ \Omega_1, \\
u_1& =& 0 &\mbox{on} \ \partial\Omega_1,\\
\end{array}\right.}\end{equation*}
where the nonnegative nonlinearity $f_1$ is continuous.

Suppose also that the set
\[Z_{1}=\left\{  x\in\Omega_{1}:\nabla u_1=0\right\}\]
is finite.

Then the extension
\begin{equation*}{\displaystyle\underline{u} (x)=\left\{\begin{array}{rl}
u_1(x),& \  \mbox{if }  \ x \in\ \overline{\Omega_1}, \\
0, &\ \mbox{if }  \ x \in\ \overline{\Omega} \setminus \Omega_1\\
\end{array}\right.}
\end{equation*}
is a sub-solution of
\begin{equation*}{\displaystyle \left\{\begin{array}{rcll}
-\Delta_pu &=& f(x,u,\nabla u) &\mbox{in} \ \Omega, \\
u& =& 0 &\mbox{on} \ \partial\Omega\\
\end{array}\right.}
\end{equation*}
for all continuous nonlinearities $f\geq 0$ such that $f_1(x,u,\nabla u)\leq f(x,u,\nabla u)$ in $\Omega_1$.
\end{lemma}
\begin{proof}This proposition is a consequence of the Divergence Theorem combined with the Hopf's Lemma (which states that $\frac{\partial u}{\partial\eta}<0$ on $\partial\Omega_1$, if $\eta$ denotes the unit outward normal to $\partial\Omega_1$, see \cite{SAKAGUCHI}, Lemma A.3).
Really, if $\phi\in C_{0}^{\infty}(\Omega)$ and $\phi\geq0$, by assuming (without loss of generality) that $Z_1=\{x_0\}$,
then
\begin{eqnarray}
\int_\Omega\left|\nabla\underline{u}\right|^{p-2}\nabla \underline{u}\cdot\nabla\phi\, dx&=&
\int_{\Omega_1\setminus B_{\varepsilon}}\left|\nabla u_1\right|^{p-2}\nabla u_1\cdot\nabla\phi\, dx\notag\\
&&\ +\int_{B_{\varepsilon}}\left|\nabla u_1\right|^{p-2}\nabla u_1\cdot\nabla\phi\, dx \label{aux1}%
\end{eqnarray}
where $B_{\varepsilon}\subset\Omega_{1}$ is a ball centered in $x_{0}$ with
radius $\varepsilon>0.$

Since $u_1\in C^2\left(\Omega\setminus B_\varepsilon\right)$ and $|\nabla u_1|>0$ in $\Omega\setminus B_\varepsilon$,
it follows from the Divergence Theorem that%
\[\int_{\partial\Omega_1\cup\partial B_\varepsilon}\phi|\nabla u_1|^{p-2}\nabla u_1\cdot {\eta}\,dS_{x}=\hspace{6.7cm}\]
\begin{align*}\hspace{1.5cm}   =&\int_{\Omega_1\setminus B_\varepsilon}\operatorname{div}\left(\phi\left|\nabla u_1\right|^{p-2}\nabla u_1\right)  dx\\
=&\int_{\Omega_1\setminus B_\varepsilon}\left|\nabla u_1\right|^{p-2}\nabla u_1\cdot\nabla\phi\, dx+
\int_{\Omega_1\setminus B_\varepsilon}\phi\operatorname{div}\left(\left|\nabla u_1\right|^{p-2}\nabla u_1\right)  dx\\
=&\int_{\Omega_1\setminus B_\varepsilon}\left|\nabla u_1\right|^{p-2}\nabla u_1\cdot\nabla\phi\,dx-\int_{\Omega_1\setminus B_\varepsilon}\phi f_1\left(x,u_1,\nabla u_1\right)\, dx.
\end{align*}

Therefore,
\[\int_{\Omega_1\setminus B_\varepsilon}\left|\nabla u_1\right|^{p-2}\nabla u_1\cdot\nabla\phi dx=\int_{\Omega_1\setminus B_\varepsilon}\phi f_1\left(x,u_1,\nabla u_1\right)dx+I_1+I_2,
\]
where%
\[I_1:=\int_{\partial\Omega_1}\phi\left|\nabla u_1\right|^{p-2}\frac{\partial u_1}{\partial\eta}\,dS_{x}\leq0
\]
and%
\[I_2:=-\int_{\partial B_\varepsilon}\phi\left|\nabla u_1\right|^{p-2}\frac{\partial u_1}{\partial\eta}\,dS_{x}.
\]

The regularity of $u_1$ implies that $\left|\left|\nabla u_1\right|^{p-2}\frac{\partial u_1}{\partial\eta}\right|=\left|\nabla u_1\right|^{p-1}\leq C$ for some positive constant $C$ which does not depend on $u_1$.
Thus,
\[\left|I_2\right|\leq C\left\|\phi\right\|_{\infty}\left|\partial
B_\varepsilon\right|\rightarrow 0\ (\text{when }\varepsilon \rightarrow 0).\]

Consequently,%
\begin{eqnarray*}
\lim_{\varepsilon\rightarrow0}\int_{\Omega_1\setminus B_\varepsilon}|\nabla u_1|^{p-2}\nabla u_1\cdot\nabla\phi\, dx&\leq&\lim_{\varepsilon\rightarrow0}\int_{\Omega_1\setminus B_\varepsilon}\phi f_1\left(x,u_1,\nabla u_1\right)  dx\\
&=&\int_{\Omega_1}\phi f_1(x,u_1,\nabla u_1)\,  dx.
\end{eqnarray*}

On the other hand,
\[\left|\int_{B_\varepsilon}\left|\nabla u_1\right|^{p-2}\nabla u_1\cdot\nabla\phi dx\right|\leq\int_{B_\varepsilon}\left|\nabla u_1\right|^{p-1}\left|\nabla\phi\right| dx\leq C\left|
\nabla\phi\right|_\infty\left|B_{\varepsilon}\right|\longrightarrow 0,\]
when $\varepsilon\to 0$.

Now, by making $\varepsilon\rightarrow0$ in (\ref{aux1}) we obtain%
\[\int_\Omega|\nabla\underline{u}|^{p-2}\nabla\underline{u}\cdot\nabla\phi\, dx\leq\int_{\Omega_1}\phi f_1(x,u_1,\nabla
u_1)\,dx\leq\int_{\Omega}\phi f(x,\underline{u},\nabla \underline{u})\,  dx.
\]

\vspace*{-.8cm}\end{proof}
\begin{remark}\rm 
(i) The hypothesis on $Z_1$ can be obtained if we suppose, for instance, $0\leq f(x,t,v)$ and, for all $t>0$, $\{(x,v)\,:\, f(x,t,v)=0\}$ is a finite set. (Of course, the more interesting case occurs when $f(x,0,v)=0$.) (ii) See H. Lou \cite{LOU} for further information on the singular set $Z_1$.
\end{remark}
\begin{proof*} From Remark \ref{rcpdd} follows that
$$k_1(\Omega_2) \leq k_1(B_\rho).$$

So, if $f$ satisfies the hypotheses (H1) and (H2), it also satisfies the hypotheses of Theorem \ref{sball}. By applying Theorem \ref{sball}, there exists a positive radial function
$u_\rho\in C^{1,\tau}\left(\overline{B_\rho}\right)$ such that
\begin{equation*}{\displaystyle \left\{\begin{array}{rcll}
-\Delta_p u_\rho &=& \omega_\rho(|x-x_0|)f(u_\rho, |\nabla u_\rho|) &\mbox{in} \ B_\rho(x_0), \\
u_\rho& =& 0 &\mbox{on} \ \partial B_\rho(x_0).\\ \end{array}\right.} \end{equation*}
Moreover, the only critical point of $u_\rho$ occurs at $x=x_0$.

It follows from Lemma \ref{lemasub} that
\begin{equation*}\underline{u}(x)=\left\{\begin{array}{ll}
u_\rho(x),& \  \mbox{if}  \ \ x \in\ B_\rho, \\
0, &\ \mbox{if}  \ x \in \Omega \setminus B_\rho\\
\end{array}\right.
\end{equation*}
is a sub-solution of problem (\ref{prob1}).

Define
\[\overline{u}=M\frac{\phi_{\Omega_2}}{\|\phi_{\Omega_2}\|_\infty}.\]

Of course, $\overline{u}\leq M$ and $\|\nabla
\overline{u}\|_\infty=M\frac{\|\nabla
\phi_{\Omega_2}\|_\infty}{\|\phi_{\Omega_2}\|_\infty}\leq
\gamma_\rho M$, by hypothesis. So, it follows from (H1) that
\[f(\overline{u},|\nabla \overline{u}|)\leq k_1(\Omega_2)M^{p-1}.\]

Moreover,
\begin{equation}\label{eq1}-\Delta_p
\overline{u}=-\Delta_p\left(M\frac{\phi_{\Omega_2}}{\|\phi_{\Omega_2}\|_\infty}
\right)=k_1(\Omega_2)M^{p-1}\|\omega\|_\infty,\end{equation}
and then, by (H2) \[ -\Delta_p \overline{u}\geq
f(\overline{u},|\nabla \overline{u}|)\,\|\omega\|_\infty \] and,
since $\overline{u}>0$ on $\partial\Omega$, $\overline{u}$ is a
super-solution of (\ref{prob1}).

Moreover, the pair $(\underline{u}, \overline{u})$ is ordered. In
fact, if $x \in \Omega \backslash B_\rho$ the result is immediate.
Otherwise we know that, $$\underline{u}=u_\rho \in C=\left\{u \in
C^1\left(\overline{B_\rho} \right): 0 \leq u\leq M, \
\textrm{and} \ \|\nabla u\|_\infty \leq \gamma_\rho M \right\},$$
and therefore, by (H1) , $f(u_\rho, |\nabla u_\rho|) \leq
k_1(\Omega_2)M^{p-1}$ and then
$$-\Delta_p \underline{u} = \omega_\rho f(u_\rho, |\nabla u_\rho|)
\leq  k_1(\Omega_2)M^{p-1}\|\omega\|_\infty=-\Delta_p \left(M
\frac{\phi_{\Omega_2}}{\|\phi_{\Omega_2}\|_\infty}
\right)=-\Delta_p \overline{u}.
$$

Moreover $$u_\rho=0 \leq M
\frac{\phi_{\Omega_2}}{\|\phi_{\Omega_2}\|_\infty} =
\overline{u}$$ on $\partial B_\rho$. We are done, since follows from the comparison principle that
$\underline{u}\leq \overline{u}$ in $B_\rho \subset \Omega$.
\end{proof*}
\section{Applications}\label{App}
In this section we choose two concrete domains $\Omega_2$ for application of Theorem \ref{maint}. In the first example, we consider a ball $B_R(x_1)=\Omega_2$ so that $\Omega\subset B_R$. In the second, we consider $\Omega_2=\Omega$ and use a result by Payne and Philippin \cite{PAYNE AND PHILIPPIN}. For this, we need to suppose that $\Omega$ is convex.

\subsection{About the constants $k_1$ and $k_2$}\label{ck1k2}
Here we establish some results about the constants $k_1(\Omega)$ and $k_2(B_\rho)$ for $B_\rho \subset \Omega$. Our comments are based in remarks made in \cite{HGZ}. First we observe that, according to Remark \ref{rcpdd}, we have $k_1(\Omega_2) \leq k_1(\Omega_1)$, if $\Omega_1\subset\Omega_2$.

Since
\begin{equation*}
k_2(B_\rho)=\displaystyle{\left[\max_{0 \leq r \leq \rho}\int_r^\rho
\left(\int_0^r
K(s,\theta)ds\right)^{\frac{1}{p-1}}d\theta\right]^{1-p}},
\end{equation*}
is easy to conclude that $k_2(B_\rho) \to \infty$ if $\rho \to 0$. Also, larger values of $\rho$ imply smaller values of $k_2(B_\rho)$ and, as we will see, also smaller values of $\gamma_\rho$. In this paper, for $x_0\in\Omega$, we choose $B_\rho$ as the largest ball centered at $x_0$ and contained in $\Omega$.

By Lemma \ref{lema comparacao entre k1 de omega e k2 da bola}, we have $k_1(\Omega) \leq k_2(B_\rho)$ for all
$B_\rho \subset \Omega$. Therefore

$$k_1(\Omega) \leq \Lambda \colon =\inf\{k_2(B_\rho): B_\rho \subset \Omega \}.$$

In the special case $\omega \equiv 1$ the constant $\Lambda$ can
be obtained since

$$k_2(B_\rho)=\displaystyle{\left[\max_{0 \leq r \leq \rho} \left( \int_t^\rho \theta^{\frac{1-N}{p-1}}d\theta \right) \left(\int_0^t
s^{N-1}ds\right)^{\frac{1}{p-1}}\right]^{1-p}}.$$

In this situation we have $$k_2(B_\rho)= \frac{C_{N,p}}{\rho^p}$$
where

\begin{equation}
 {\displaystyle C_{N,p}= \left\{
\begin{array}{lll}
 \frac{p^p}{(p-1)^{p-1}}e^{p-1} & \textrm{if}  &  N=p;\\
\frac{N^p}{(p-1)^{p-1}} \left( \frac{p}{N} \right)^{\frac{p(p-1)}{p-N}} & \textrm{if}   & N \neq p.\\
\end{array}\right.}
\end{equation}
In this case,
$$\Lambda=\frac{C_{n,p}}{\rho_*^p}, \ \textrm{ where} \ \rho_*:=\sup \{\rho: B_\rho(x_0) \subset \Omega, x_0 \in \Omega \}.$$

\subsection{Radial Supersolution}\label{SSRadial}
For all $x\in\Omega$, let $d(x)={\rm dist}(x,\partial \Omega)$. We denote by $r_*=\sup_{x\in\Omega}d(x)$. Let $B_{r_*}$ be a ball with center at $x_0\in\Omega$ such that $B_{r_*}\subset\Omega$.

Choose $R$ such that $\Omega\subset B_R$, where $B_R$ is a ball with center at $x_1\in\Omega$ and
let $\phi_R \in C^{1,\alpha}(\overline{B_R(x_1)})\cap W_0^{1,p}(B_R(x_1))$ be the unique positive solution of
\begin{equation}{\displaystyle \left\{\begin{array}{rcll}
-\Delta_p \phi_R &=&\|\omega\|_\infty &\mbox{in} \ B_R(x_1), \\
\phi_R& =& 0 &\mbox{on} \ \partial B_R(x_1), \end{array}\right.}\label{tcp}
\end{equation}
and consider the positive constant $k_1(\Omega_2)=k_1(B_R)=\|\phi_R\|_{\infty}^{-(p-1)}$.

We define, as in Theorem \ref{maint},
$$\overline{u}:= M\frac{\phi_R}{\ \|\phi_R\|_\infty} \in C^{1,\alpha}\left( \overline{B_R(x_1)}\right)\cap W_0^{1,p}(B_R(x_1)).$$

Of course, $0<\overline{u}\leq M$. According to Section \ref{SRad}, we have
\begin{eqnarray}\label{estquo}
\phi_R(r) & = & \int_r^R \varphi_q \left(\int_0^\theta \left(\frac{s}{\theta}\right)^{N-1}\|\omega\|_\infty\, ds \right)d\theta\notag\\
&=& \|\omega\|_\infty^{\frac{1}{p-1}}\int_r^R \varphi_q \left(\frac{1}{{\theta}^{N-1}}\int_0^\theta {s}^{N-1}
 ds \right)d\theta\notag\\
&=& \|\omega\|_\infty^{\frac{1}{p-1}}\int_r^R\left(\frac{\theta}{N}\right)^{\frac{1}{p-1}}d\theta\notag\\
& =& \frac{p-1}{p} \left(\frac{\|\omega\|_\infty}{N}\right)^{\frac{1}{p-1}}\left(R^{\frac{p}{p-1}}-r^{\frac{p}{p-1}} \right).
\end{eqnarray}

On the other hand, we have $\nabla \phi_R(x)=\phi'_R(r)
\frac{x-x_0}{r}$, from what follows $|\nabla \phi_R(x)|=|\phi'_R(r)|$. Thus,
$$\|\nabla \phi_R\|_\infty=|\phi'_R(R)|=\left(\int_0^R\left(\frac{s}{R}\right)^{N-1}\|\omega\|_\infty ds\right)^{\frac{1}{p-1}}
=\left(\frac{\|\omega\|_\infty}{N}\right)^{\frac{1}{p-1}}R^{\frac{1}{p-1}}$$
and
\begin{equation}\label{q/R}\frac{\|\nabla \phi_R \|_\infty}{\|\phi_R\|_\infty}= \frac{p}{p-1} R^{\frac{1}{p-1}-\frac{p}{p-1}}=\frac{q}{R}.\end{equation}
So, we need to choose $\rho>0$ such that $B_\rho\subset\Omega$ and
$$\frac{q}{R}<\gamma_\rho,$$
in order to have
$$0 \leq |\nabla \overline{u}| = M \frac{|\nabla \phi_R |}{\|\phi_R\|_\infty} \leq M \frac{\|\nabla \phi_R \|_\infty}{\|\phi_R\|_\infty}= \frac{q}{R}M \leq \gamma_\rho M.$$

To choose $\rho$, let us consider the possibilities
\begin{enumerate}
\item [$(i)$] $r_*\leq \frac{R}{q}\ (<R)$.

We choose $\rho=r_*$, because
\[\frac{\|\nabla \phi_R\|_\infty}{\|\phi_R\|_\infty}=\frac{q}{R}\leq \frac{1}{r_*}=\frac{1}{\rho}\leq\gamma_{\rho}.\]
\item [$(ii)$] $\frac{R}{q}\leq r_*\ (<R)$.

We choose $\rho=\frac{R}{q}$, since
\[\frac{\|\nabla \phi_R\|_\infty}{\|\phi_R\|_\infty}=\frac{q}{R}=\frac{1}{\rho}\leq\gamma_{\rho}.\]
\end{enumerate}


In the special case $\omega_\rho\equiv 1$, we can always choose $\rho=r_*$, since
\[\frac{\|\nabla \phi_R\|_\infty}{\|\phi_R\|_\infty}=\frac{q}{R}\leq \frac{q}{r_*}=\frac{q}{\rho}=\gamma_{\rho}.\]
This value of $\rho$ corresponds to the smallest values of $k_2(B_\rho)$ and $\gamma_\rho$. The best value for $k_1(B_R)$ is obtained when $R$ is the smallest radius such that $B_R(x_1)\supset\Omega$ for $x_1\in\Omega$.
\subsection{Applying a maximum principle of Payne and Phillipin}
If we choose $\Omega_2=\Omega$, we need to suppose that $\Omega$ is convex to control the quotient (\ref{quo}). For this, we consider  the torsional creep problem
\begin{equation}\label{torsionalcreep}
{\displaystyle \left\{\begin{array}{rcll}
-\Delta_p \psi_\Omega &=& 1 &\mbox{in} \ \Omega, \\
\psi_\Omega& =& 0 &\mbox{on} \ \partial\Omega.\\
\end{array}\right.}
\end{equation}
For more information on the torsional creep problem, see Kawohl \cite{KAWOHL}.

In order to estimate the quotient (\ref{quo}),  we state a maximum principle of Payne and Philippin \cite{PAYNE AND PHILIPPIN}:
\begin{theorem}\label{teorema de payne e philippin}
Let $\Omega \subset \mathbb{R}^N$ be a convex domain such that $\partial\Omega$ is a $C^{2,\alpha}$ surface. If $u=const.$ on $\partial \Omega$, then
\begin{equation}\label{epp}
\Phi(x)= 2\frac{p-1}{p}|\nabla \psi_\Omega|^p+2\psi_\Omega
\end{equation}
takes its maximum value at a critical point of $\psi_\Omega$.
\end{theorem}

\begin{lemma}\label{lema quociente} If $\Omega$ satisfies $({H}4)$, then
$$\|\nabla \psi_\Omega\|_\infty \leq \left(q\|\psi_\Omega\|_\infty\right)^{\frac{1}{p}},$$
what yields
$$\frac{\|\nabla \psi_\Omega\|_\infty}{\|\psi_\Omega\|_\infty} \leq \frac{q^{\frac{1}{p}}}{\|\psi_\Omega\|_\infty^{\frac{1}{q}}}.$$
\end{lemma}
\begin{proof} By Theorem \ref{teorema de payne e philippin}, $\Phi$ takes its maximum value
at a point where $\nabla \psi_\Omega=0$. So, it follows from (\ref{epp}) that
\begin{equation*}
2\frac{p-1}{p}|\nabla \psi_\Omega|^p+2\psi_\Omega \leq  2\|\psi_\Omega\|_\infty.
\end{equation*}
Therefore
\begin{equation*}
|\nabla \psi_\Omega|^p \leq \frac{p}{p-1} \left(\|\psi_\Omega\|_\infty -\psi_\Omega(x)
\right) \leq \frac{p}{p-1} \|\psi_\Omega\|_\infty  = q \|\psi_\Omega\|_\infty, \ \forall\
x \in \Omega,
\end{equation*}
thus producing
\begin{equation*}
|\nabla \psi_\Omega| \leq \left( q \|\psi_\Omega\|_\infty\right)^{\frac{1}{p}}, \
\forall\ x \in \Omega.
\end{equation*}

But the $p$-Laplacian is degenerated at the origin. So, in order to estimate the quotient (\ref{quo}), a regularization of $-\Delta_p$ is done by considering, as in Sakaguchi \cite{SAKAGUCHI}, the problem
\begin{equation*}
{\displaystyle \left\{\begin{array}{rcll}
- \textrm{div}((\varepsilon + |\nabla \phi_\varepsilon|^2)^{\frac{p-2}{2}}\nabla \phi_\varepsilon ) &=& 1 &\mbox{in} \ \Omega, \\
\phi_\varepsilon& =& 0 &\mbox{on} \ \partial\Omega.\\
\end{array}\right.}\end{equation*}
Sakaguchi proves that the solution $\phi_\varepsilon$ converges uniformly to $\psi_\Omega$ as $\epsilon\to 0$.
The regularization permits us to estimate (\ref{quo}) in the case of the torsional creep problem  (\ref{torsionalcreep}):
\begin{equation*}
\frac{\|\nabla \psi_\Omega\|_\infty }{\| \psi_\Omega\|_\infty}\leq  \frac{q^{\frac{1}{p}}}
{\|\psi_\Omega\|_\infty^{1-\frac{1}{p}}}= \frac{q^{\frac{1}{p}}}
{\|\psi_\Omega\|_\infty^{\frac{1}{q}}}.
\end{equation*}

\vspace*{-.7cm}\end{proof}\vspace{.5cm}

An immediate consequence of Lemma \ref{lema quociente} is an
estimate of the quotient (\ref{quo}) in the case $\Omega=\Omega_2$: we have


\begin{equation}\label{quoc Omega2}
\frac{\|\nabla \phi_\Omega\|_\infty}{\|\phi_\Omega\|_\infty}  \leq
\frac{\left( q \|\omega\|_\infty
\right)^\frac{1}{p}}{\|\phi_\Omega\|_\infty^\frac{1}{q}}.
\end{equation}

 In fact, if $\psi_\Omega$ is a solution of the torsional creep problem (\ref{torsionalcreep}), then
$$\phi_\Omega=\|\omega\|_\infty ^{\frac{1}{p-1}}\psi_\Omega$$
is a solution of (\ref{pss}). So,
\begin{equation}
\frac{\|\nabla \phi_\Omega\|_\infty}{\|\phi_\Omega\|_\infty}= \frac{\|\nabla
\psi_\Omega\|_\infty}{\|\psi_\Omega\|_\infty} \leq
\frac{q^{\frac{1}{p}}}{\|\psi_\Omega\|_\infty^{\frac{1}{q}}}=\frac{q^{\frac{1}{p}}}{\|\phi_\Omega\|_\infty^{\frac{1}{q}}}
\|\omega\|_\infty ^{\frac{1}{p}} = \frac{{(q \|\omega\|_\infty
)}^{\frac{1}{p}}}{\|\phi_\Omega\|_\infty^{\frac{1}{q}}}.
\end{equation}

We observe that the quotient (\ref{quo}) was controlled for any convex domain $\Omega_2\supset\Omega$.

As in the Subsection \ref{SSRadial}, let $B_{r_*}$ be a ball with larger radius such that $B_{r_*}\subset\Omega$. We consider the solution $\phi_*$ of the problem
\begin{equation*}{\displaystyle \left\{\begin{array}{rcll}
-\Delta_p \phi_*&=& \|\omega\|_\infty &\mbox{in} \ B_{r_*}, \\
\phi_*& =& 0 &\mbox{on} \ \partial B_{r_*}.\\
\end{array}\right.}
\end{equation*}
Since $B_{r_*}\subset \Omega$, from the comparison principle follows that
\[\|\phi_*\|_\infty\leq \|\phi_\Omega\|_\infty.\]

But
\begin{eqnarray*}
\|\phi_*\|_\infty&=&\int_0^{r_*}\left(\int_0^\theta \left(\frac{s}{\theta}\right)^{N-1}\|\omega\|_\infty\,ds\right)^{\frac{1}{p-1}}d\theta\\
&=&\left(\frac{\|\omega\|_\infty}{N}\right)^{\frac{1}{p-1}}\int_0^{r_*}\theta^{\frac{1}{p-1}}\,d\theta=
\left(\frac{\|\omega\|_\infty}{N}\right)^{\frac{q}{p}}\frac{r_*^q}{q},
\end{eqnarray*}
thus yielding
\[\frac{\|\nabla \phi_\Omega\|_\infty}{\|\phi_\Omega\|_\infty}\leq \frac{(q\|\omega\|_\infty)^{\frac{1}{p}}}{\|\phi_\Omega\|_\infty^{\frac{1}{q}}}\leq \frac{q^{\frac{1}{p}+\frac{1}{q}}}{r_*}\|\omega\|_\infty^{\frac{1}{p}}
\left(\frac{N}{\|\omega\|_\infty}\right)^{\frac{1}{p}}=\sqrt[p]{N}\frac{q}{r_*}.\]

We now choose $\rho$ given by
\[\rho=\frac{r_*}{q\sqrt[p]{N}}=\frac{p-1}{p\sqrt[p]{N}}r_*\ (< r_*).\]

Then, 
we have
\[\frac{\|\nabla\phi_\Omega\|_\infty}{\|\phi_\Omega\|_\infty}\leq\frac{1}{\rho}=\gamma_{\rho}.\]

In the special case $\omega\equiv 1$, we can take $\rho$ such that
\[\frac{q}{\rho}=\frac{q\sqrt[p]{N}}{r_*},\]
since $\gamma_\rho=q/\rho$.

Thus, we have
\[\rho=\frac{r_*}{\sqrt[p]{N}}<r_*\]
and
\[\frac{\|\nabla \phi_\Omega\|_\infty}{\|\phi_\Omega\|_\infty}\leq \frac{q}{\rho}=\gamma_{\rho}.\]

\section{Examples}\label{Ex}
In this section, we start by studying the problem
\begin{equation} \label{ex1}
{\displaystyle \left\{\begin{array}{rcll}
-\Delta_pu &=& \lambda u(x)^{q-1}(1+|\nabla u(x)|^p) &\mbox{em} \ \Omega, \\
u& =& 0 &\mbox{em} \ \partial\Omega,\\
\end{array}\right.}\end{equation}
where $\Omega$ is a smooth, bounded domain in $\mathbb{R}^N$, $1<q<p$, and $\lambda$ a positive parameter. Problem (\ref{ex1}) is sublinear at the origin.

The solution of this example will permit us to solve
\begin{equation*}
{\displaystyle \left\{\begin{array}{rcll}
-\Delta_pu &=& \lambda f(u,|\nabla u(x)|) &\mbox{em} \ \Omega, \\
u& =& 0 &\mbox{em} \ \partial\Omega,\\
\end{array}\right.}\end{equation*}
for a class of nonlinearities $f$ that are superlinear both at the origin and at $+\infty$.

Results of Huang (\cite{HUANG}) guarantee the unicity of solutions in $$\displaystyle{\Gamma_q=\left\{ u \in W_{0}^{1,p}(\Omega): \int_\Omega |u|^q=1\right\}}$$ for the problem
\begin{equation*} {\displaystyle \left\{\begin{array}{rcll}
-\Delta_pu &=& \lambda u(x)^{q-1} &\mbox{in} \ \Omega, \\
u& =& 0 &\mbox{on} \ \partial\Omega.\\
\end{array}\right.}\end{equation*}

In Montenegro and Montenegro (\cite{Montenegros}) degree theory and the method of sub- and supersolutions are combined to present conditions for existence and nonexistence of weak, positive solutions for the problem
\begin{equation*}\label{prob montenegro}
{\displaystyle \left\{\begin{array}{rcll}
-\Delta_p u &=& \frac{a}{1+ku}|\nabla u|^p  +b(1+ku)^{p-1}&\mbox{in} \ \Omega, \\
u& =& 0 &\mbox{on} \ \partial\Omega,\\
\end{array}\right.}
\end{equation*}
where $a$ e $b$ are positive constants and $k \geq 0$. 

Also, in Iturriaga, Lorca and Sanchez (\cite{iturriaga sanchez lorca}) no qual os autores consideram o problema

\begin{equation}\label{prob iturriaga}
{\displaystyle \left\{\begin{array}{rcll}
-\Delta_p u &=& \lambda f(x,u)+|\nabla u|^p &\mbox{in} \ \Omega, \\
u& =& 0 &\mbox{on} \ \partial\Omega,\\
\end{array}\right.}
\end{equation}
where $\lambda$ is a positive parameter and $f(x,u)$ a Caratheodory function such that
$$c_0 u^{q-1} \leq f(x,u) \leq c_1 u^{q-1}, \ \textrm{para todo} \ (x,t) \in \overline{\Omega} \times [0,\infty)$$
for positive constants $c_0,c_1$. However, the problem was solved by applying a change of variables that transforms (\ref{prob iturriaga}) into a problem that does not depend on the gradient.

Under similar but different hypotheses, problem (\ref{ex1}) was solved in \cite{HGZ2}. Here, we show that this problem has a positive solution for each $\lambda\in (0,\lambda^*]$, where $\lambda^*$ will be given in the sequence. To solve problem (\ref{ex1}) we consider, as in Subsection \ref{SSRadial}, $B_\rho$ as the largest open ball contained in $\Omega$ and $B_R$ such that $\Omega \subseteq B_R$.

\begin{remark}\rm The nonlinearity $\lambda u(x)^{q-1}(1+|\nabla u(x)|^p)$ is related to the weight function $\omega(x) \equiv 1$. So, the constants in hypotheses (H1) e (H2) are given by
\begin{equation}\label{k1 BR}k_1:=k_1(B_R)=\|\phi_R\|_{\infty}^{-(p-1)}=\left(\frac{p-1}{p}\right)^{1-p} {N}R^{-p},\end{equation}

\begin{equation}\label{k2 B_rho}
{\displaystyle
k_2:= k_2(B_\rho)=
\left\{\begin{array}{rcll}
\displaystyle{ \left[
\frac{p-1}{p} \left(\frac{p}{N}\right)^{\frac{N}{N-p}} \right]^{1-p} \frac{N}{\rho^{p}}}, &\mbox{if} \ N \neq p, \\
\\
\displaystyle{ \left(\frac{p-1}{e p}\right)^{1-p} \frac{p}{\rho^p}},   &\mbox{if} \ N=p,\\
\end{array}\right.}
\end{equation}
and
\begin{equation}\gamma = \gamma_\rho =\frac{p}{p-1} \frac{1}{\rho}.\end{equation}

From now on, $k_1$ and $k_2$ denote the constants (\ref{k1 BR}) and (\ref{k2 B_rho}), respectively. According to Lemma \ref{lema comparacao entre k1 de omega e k2 da bola}, we have $k_1<k_2$.
$\hfill\lhd$\end{remark}

Of course, the nonlinearity $\lambda u(x)^{q-1}(1+|\nabla u(x)|^p)$ satisfies (H3) for any value of $\lambda$. We will obtain constants $\delta,M$ (with $0<\delta<M$) such that hypotheses (H1) and (H2) of Theorem \ref{maint} are verified.

To satisfy (H1), $M$ must be chosen such that
\begin{equation}\label{cond para valer H1}
\lambda M^{q-1}(1+(\mu M)^p)\leq \alpha M^{p-1}.\end{equation}

So, defining the function $H\colon [0,\infty)\to [0,\infty]$ by $H(M)=M^{q-p}(1+\mu^p M^p)$, the last inequality is equivalent to
\[H(M)\leq \frac{\alpha}{\lambda}.\]
We have \[\lim_{M\to 0^+}H(M)=\infty=\lim_{M\to\infty}H(M),\]
and the function $H$ has a unique critical point $M_*$, given by
\[\mu^p M^p_*=\frac{p}{q}-1,\]
where $H$ assumes its minimum value
\[H(M_*)=M^{q-p}(1+\mu^pM_*^p)=\frac{1}{\mu^{q-p}}\left(\frac{p}{q}-1\right)^{\frac{q-p}{p}}\left(\frac{p}{q}\right).\]

\begin{figure}[tbh]
\begin{center}
{\includegraphics[height=4.5cm]{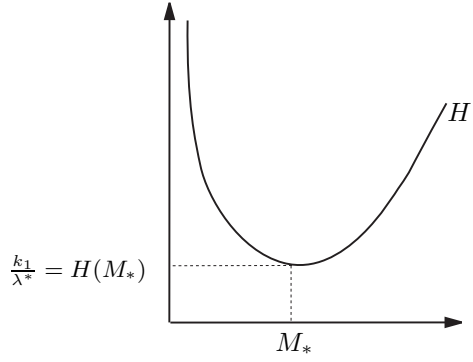}}
\begin{picture}(0,0)
\put(-12,80){$H$} \put(-77,-7){$M_*$} \put(-178,21){\small{$\frac{k_1}{\lambda^*}= H(M_*)$}}
\end{picture}
\label{figf}
\end{center}
\caption{The graph of $H$ assumes it minimum at the point $M^*$.}%
\end{figure}

Taking $M := M_*$ in (\ref{cond para valer H1}) and defining
\begin{equation}\label{lambda_*}
\lambda^*=\frac{k_1}{H(M_*)},
\end{equation}
hypothesis (H1) is verified for any $0 < \lambda \leq \lambda^*$. The choice $M=M_*$ makes $\lambda^*$ to be the best possible value of the parameter such that Theorem \ref{maint} guarantees the existence of a positive solution for problem (\ref{ex1}).

Now, for any fixed $\lambda\in (0,\lambda^*]$, we try to verify (H2). More precisely, we look for $\delta_\lambda:=\delta(\lambda)$ such that
\[\lambda u^{q-1}(1+|\nabla u|^p)\geq k_2\delta_\lambda^{p-1},\quad \delta_\lambda\leq u\leq M_*,\ \ 0\leq |\nabla u|\leq \gamma M_*.\]

For this, we consider the function $G\colon(0,\infty) \rightarrow [0,\infty)$ given by
\begin{equation}\label{G} G(x)=x^{q-p}.
\end{equation}

We clearly have $G(x) \leq H(x)$ for any $x \in (0, \infty)$ and (H2) is verified if
\begin{equation} \label{H aux}
\lambda G(\delta_\lambda) \geq k_2,\end{equation}
that is,
\begin{equation} \label{H aux2}
\delta_\lambda\leq \left( \frac{\lambda}{k_2}\right)^{\frac{1}{p-q}}.\end{equation}

\begin{figure}[htb]
\begin{center}
{\includegraphics[height=4.5cm]{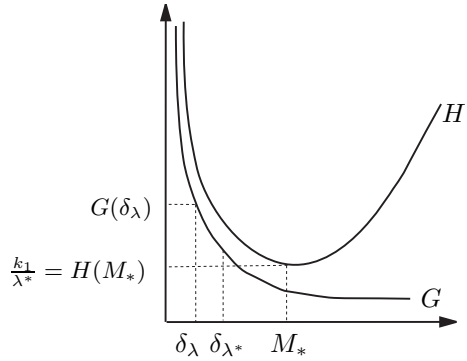}}
\begin{picture}(0,0)
\put(-12,80){$H$}
\put(-20,10){$G$}
\put(-77,-7){$M_*$}
\put(-113,-7){$\delta_\lambda$}
\put(-100,-7){$\delta_{\lambda^*}$}
\put(-176,21){\small{$\frac{k_1}{\lambda^*}=H(M_*)$}}
\put(-145,45){\small{$
G(\delta_\lambda)$}}
\end{picture}
\label{figfg}
\end{center}
\caption{The graphs of $H$ and $G$.}%
\end{figure}

So, for any $\lambda\in (0,\lambda^*]$, (H2) is satisfied if we take $\delta_\lambda>0$ verifying the above inequality. Observe that the same value of $\delta_\lambda$ is valid for any $\tilde\lambda\in [\lambda,\lambda^*]$.

Since $0<\lambda\leq \lambda^*$, the largest value of $\delta_\lambda$ is attained at $\lambda^*$. So, the condition  $\delta_\lambda<M_*$ always holds:
\[\delta_\lambda\leq\left(\frac{\lambda^*}{k_2}\right)^{\frac{1}{p-q}}\leq \left(\frac{\lambda^*}{k_1}\right)^{\frac{1}{p-q}}
=\left(\frac{1}{H(M_*)}\right)^{\frac{1}{p-q}}
=M_*\left(\frac{q}{p}\right)^{\frac{1}{p-q}}<M_*.\]

\begin{remark}\label{obs theta} \rm The existence of positive solutions for the problem
\begin{equation}
{\displaystyle \left\{\begin{array}{rcll}
-\Delta_pu &=& \lambda \ \omega(x) u(x)^{q-1}(1+|\nabla u(x)|^p) &\mbox{in} \ \Omega, \\
u& =& 0 &\mbox{on} \ \partial\Omega,\\
\end{array}\right.}\end{equation}
where $\Omega\subset\mathbb{R}^N$ is a bounded, smooth domain and $1<q<p$ follows analogously for any continuous weight
function $\omega\colon \overline{\Omega} \rightarrow \mathbb{R}$. We can also change $p$ for any value $0<\theta<p$.\end{remark}

The main advantage of the method that leads to Theorem \ref{maint} are the hypotheses (H1) and (H2).
\begin{figure}[tbh]
\begin{center}
{\includegraphics[height=4.5cm]{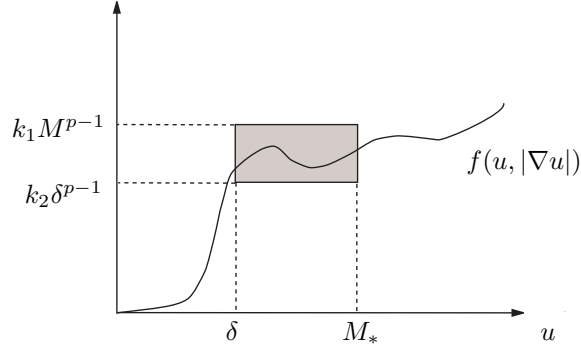}}
\begin{picture}(0,0)
\put(-45,60){$f(u, |\nabla u|)$}
 \put(-212,47){$k_2\delta^{p-1}$}
\put(-217,72){$k_1M^{p-1}$} \put(-136,-5){$\delta$}
 \put(-17,-5){$u$}\put(-92,-5){$M_*$}
\end{picture}
\label{figura1: box}
\end{center}
\caption{Example of a nonlinearity $f$, superlinear at the origin and satisfying (H1) e (H2). O graph illustrates the perspective $|\nabla u|$= constant.}
\end{figure}

With the exception of (H3), no other assumption on the nonlinearity $f$ is necessary. So, $f$ can be superlinear both at the origin and at $+\infty$.

\begin{remark} \rm  The radial problem
\begin{equation}\label{ex anel}
\displaystyle \hspace*{-.2cm}\left\{\!\!\!\!\begin{array}{rcll}
-\Delta_pu\!\!&=&\!\! \lambda \ \omega(|x|) u(x)^{r}(1+|\nabla u(x)|^\theta), &\mbox{if} \ R_1 < |x| < R_2, \\
u\!\!& =&\!\! 0, &\mbox{if} \ |x|=R_1 \ \mbox{or} \ |x|=R_2,\\
\end{array}\right.
\end{equation}
where $\lambda=1$, $r>p-1$, $0 \leq \theta \leq p$ and $\omega\colon [R_1, R_2] \rightarrow [0, \infty)$ is a continuous function not equal zero was considered in \cite{figubilla}. In that work, the existence of a positive solution for (\ref{ex anel}) was obtained as a consequence of Krasnosel'skii Fixed Point Theorem for mappings defined in cones.

Our result complements those obtained in \cite{figubilla} and guarantees the existence of a solution for (\ref{ex anel}) in the case $0<r<p-1$ and $0< \theta \leq p$.

In fact, if $\lambda \leq \lambda^*$, the existence of a positive solution follows from Theorem \ref{maint}. Since $\omega\neq 0$, there exists $B_\rho$ contained in the domain where $\omega > 0$. In $B_\rho$ we obtain a subsolution  $\underline{u}$ of (\ref{ex anel}).

The inclusion of the parameter $\lambda$ is necessary because of the hypotheses (H1) and (H2). In the particular case  $\lambda=1$, the existence of solution is obtained only if $1 \leq \lambda^*$.
\end{remark}

\section{Appendix}
In this appendix we give a direct prove that the radial operator $T$ of Section \ref{SRad} is continuous and compact.

In fact, compactness of $T$ can be obtained from Arzel\'{a}-Ascoli and
Dominated Convergence Theorems.

Let $\left\{u_m \right\}_{m\in \mathbb{N}}$ a bounded sequence in
$X$, $\|u_m\| \leq M$. Is this case, $\left\{Tu_m\right\}$ and
$\left\{(Tu_m)'\right\}$ are bounded in $X$ and uniformly
equicontinuous.


In fact, let $C=\displaystyle{\sup_{0\leq t,s \leq M}f(s,t)}$. Since $K(s,\theta)=\left(\frac{s}{\theta}\right)^{N-1}\omega_\rho(s)\leq\omega_\rho(s)$ and $0\leq\theta \leq \rho$, we have

$$Tu_m(r) \leq \int_0^\rho \varphi_q \left(C\int_0^\theta K(s,\theta)\,ds\right) d\theta \leq \rho\varphi_q \left( C \int_0^\rho \omega_\rho(s)ds\right),$$ showing that $Tu_m(r)$ is
uniformly bounded in the sup norm.

We also have,
\begin{eqnarray*}|(Tu_m)'(r)|&=&\varphi_q\left(\int_0^rK(s,r)f(u_m(s),|\nabla u_m(s)|)\,ds\right)\\
&\leq&  \varphi_q\left(C \int_0^\rho \omega_\rho(s)ds\right),
\end{eqnarray*}
proving that $\left\{ Tu_m\right\}$ is equicontinuous.

Let us prove that $\left\{(Tu_m)'\right\}$ is equicontinuous.
Deriving $(Tu_m)'(r)$ we have

\[|(Tu_m)''(r)|=\!\left(\!\frac{1}{p-1}\right)\!\!\left[|v'_m(r)|^{2-p}
w_\rho(s) f(u_m(s), |u'_m(s)|)+ \left(\!\frac{n-1}{r}\!\right)|v'_m(r)|
\right],\]
where $v_m=(Tu_m)'$.

If $1<p \leq 2$, the right-hand side of the last equality is bounded, thus $(Tu_m)'$ is Lipschitz-continuous and, consequently,
equicontinuous.

If $p>2$, $|(Tu_m)'(r)|$ is H\"{o}lder continuous with exponent $\frac{1}{p-1}$ (and
consequently $|(Tu_m)(r)|$ is equicontinuous). In fact,
$(Tu_m)'(r)= (\varphi_q \circ \lambda_m)(r)$, where
$$\varphi_q(x)=x^{\frac{1}{p-1}}$$
is a H\"{o}lder continuous function and
$$\lambda_m(r):=\int_0^r K(s,r)f(u_m(s),|\nabla u_m(s)|)\,ds.$$ We claim that $\lambda_m$ is locally Lipschitz continuous, uniformly on $m$. For this, we note that
$\lambda_m \in C^1([0,R])$ with
$$\displaystyle{\lim_{r\rightarrow 0^+}\lambda_m(r)=0}$$
and

$$\displaystyle{\lim_{r\rightarrow 0^+}\lambda_m'(r)=\frac{\omega(0) f(u_m(0),|u'_m(0)|)}{N}}.$$
Therefore, the Mean Value Theorem guarantees the existence of $L>0$ such that

$$|\lambda_m(r) - \lambda_m(t)| \leq L|r-t|,$$
proving our claim.

Since $(Tu_m)'(r)= \varphi_q \circ \lambda_m(r)$ we can conclude the
equicontinuity of $\left\{ (Tu_m)'\right\}$

\begin{align*}
\left|(Tu_m)'(r)-(Tu_m)'(t)\right|&\leq|\varphi_q(\lambda_m(r))-\varphi_q(\lambda_m(t))|\\
&\leq|\lambda_m(r)-\lambda_m(t)|^{\frac{1}{p-1}} \leq
L|r-t|^{\frac{1}{p-1}}.
\end{align*}


We also note that, if $\left\{u_m \right
\}_{m \in \mathbb{N}}$ converges uniformly to $u$ in $[0,R]$, then $Tu_{m_j} \rightarrow Tu$ for all the subsequences $\left\{u_{m_j}\right \}$  of $\left\{u_{m}\right\}$, by the Dominated Convergence Theorem. From this follows that $T$ is continuous. $\hfill\Box$


\end{document}